\documentclass[11pt]{article}
\usepackage[top=3.5cm,left=1.9cm,right=1.9cm, bottom=4.5cm]{geometry}
\usepackage[utf8]{inputenc}
\usepackage{graphicx,xcolor,amsthm}
\usepackage[english]{babel}
\usepackage{amssymb,mathrsfs,amsmath}
\usepackage[us, 12hr]{datetime}
\usepackage{enumerate}
\usepackage{relsize}
\usepackage{xcolor}

\usepackage{pgfplots}
\usepackage{tikz}
\usepackage{caption}
\usepackage{subcaption} 
\usepackage{float}
\usepackage{graphicx}
\usepackage{dsfont}
\usepackage{latexsym}

\usepackage{enumitem}

\newcommand{\R}{\mathbb{R}}

\newcommand{\Z}{\mathbb{Z}}

\renewcommand{\phi}{\varphi}

\newcommand{\abs}[1]{\left\lvert#1\right\rvert}

\newtheorem{theorem}{Theorem}
\newtheorem{prop}[theorem]{Proposition}
\newtheorem{lemma}[theorem]{Lemma}

\theoremstyle{definition}
\newtheorem{definition}[theorem]{Definition}

\theoremstyle{remark}
\newtheorem{remark}[theorem]{Remark}

\DeclareMathOperator{\EX}{\mathbb{E}}

\graphicspath{../../}
\title{Fractal dimension, approximation and data sets\footnote{The project was supported in part by the funds provided by the University of Rochester Office of Undergraduate Research. Betti, Chio, Fleischman, Iosevich, Iulianelli, Martino, Pack, Sheng, Taliancic, Whybra, Wyman and Zhao were affiliated with the University of Rochester when this paper was written. Kirila is affiliated with Parker Avery, Mayeli with CUNY, Thomas with Cornell, and Yildirim with Google.}}
\author{L. Betti, I. Chio, J. Fleischman, A. Iosevich\footnote{The work of the fourth listed author was supported in part by the NSF HDR Tripods 1934962 grant, and the NSF DMS 2154232 grant.}, F. Iulianelli, S. Kirila, M. Martino,  A. Mayeli, \\ S. Pack, Z. Sheng, C. Taliancic, A. Thomas, N. Whybra, E. Wyman\footnote{The work of the fourteenth listed author was supported in part by the NSF DM 2204397 grant.}, U. Yildirim, K. Zhao}

\date{\today}

\begin{document}

\maketitle

\abstract{The purpose of this paper is to study the fractal phenomena in large data sets and the associated questions of dimension reduction. We examine situations where the classical Principal Component Analysis is not effective in identifying the salient underlying fractal features of the data set. Instead, we employ the discrete energy, a technique borrowed from geometric measure theory, to limit the number of points of a given data set that lie near a $k$-dimensional hyperplane, or, more generally, near a set of a given upper Minkowski dimension. Concrete motivations stemming from naturally arising data sets are described and future directions outlined.}

\section{Introduction}

\vskip.125in 

The age of big data is firmly upon us. It is difficult to imagine even a medium sized company today that does not routinely analyze a million or so points in a thousand dimensional space in order to maintain its bottom line. Such analysis would be both painful and inefficient if it had not been for the emergence of neural network models, the concept first proposed by Alexander Bain, a Scottish philosopher, and William James, an American psychologist. Bain believed (\cite{B1873}) that every activity led to the firing of specific neurons, and the connections became more pronounced when activities were repeated. This process led to the formation of memory. James proposed a similar scheme (\cite{J1890}), except that he suggested that memories and actions were formed by electrical currents flowing among the neurons in the brain, and the individual neural connections were not necessary. 

A big breakthrough in neural network research took place in 1958 when Frank Rosenblatt, an American psychologist, created the perceptron (\cite{R1958}), an algorithm for pattern recognition based on a two-layer learning computer network using addition and subtraction. With the invention of back-propagation by Paul Werbos (\cite{W1975}), the neural network theory starting taking on its modern form where both practical and theoretical problems were awaiting an eager group of researchers from a variety of fields of learning. 

One of the basic results in neural network theory is the Universal Approximation Theorem (see e.g. \cite{DS14}), which establishes, for a variety of classes of a function that if $f: {[-1,1]}^{d-1} \to [-1,1]$ satisfying a set of natural conditions, then for any $\epsilon>0$ there exists a feed forward neural network $N(x)$ such that $|f(x)-N(x)| \leq \epsilon$. In particular, such a result holds if $f$ is a Lipschitz function, and the number of neurons in the three layer feed forward neural network depends on the Lipschitz constant. But what happens if the function to be approximated is highly oscillatory? In particular, what happens if the graph of this function has dimension $>d-1$? 

This leads us to one of the basic questions in data science, which is to determine the "effective" dimension of a large data set. If a data set has $10^6$ points
in $1000$-dimensional space, it is extremely useful to be able to detect if a significant proportion of these points live on a lower dimensional plane, or a more complicated surface, reflecting the hidden relationships between the features present in the data set. One of the main tools in this area is the Principal Component Analysis (PCA) (see e.g. \cite{DFO2020}), This method has been and will remain a fundamental tool in the study of dimensionality of data sets. We propose to complement this method with a set of tools that capture important dimensionality phenomena that PCA does not see. 

As a simple synthetic example, let us use the stages of construction of a Cantor type subset of $[0,1]$ consisting of real numbers that have only $0$s and $2$s in their base $4$ expansions. More precisely, divide $[0,1]$ in four equal segments, and remove the segments $(1/4,1/2)$ and $(3/4,1]$. We keep the endpoints of the remaining segments, obtaining $\{0,1/4,1/2, 3/4\}$. Repeating the same procedure with the remaining two intervals of length $1/4$, we obtain $16$ points at the next stage, and so on. Let $C_{2^{k+1}}$ denote the resulting set and note that it has $2^{k+1}$ points. Let $n=2^{2(k+1)}$ and define $P_n=C_{2^{k+1}} \times C_{2^{k+1}}$. From the point of view of PCA this set is two-dimensional, and yet the fractal pattern that is present may be of considerable practical significance. We are going to explore this idea from a numerical point of view below. 

In order to illustrate the ubiquity of fractional dimension phenomena in large data sets, let's consider the following hypothetical example. Suppose that we have a time series representing sales of a retail store going back $40$ years. Suppose that one wished to look at the times when the sales were in the top $5 \%$ of all sales in a given year, and it turned out that this happened every July, December and April, and that during those months it happened during the first week, and during that week it happened on Fridays and Saturdays, and that on those days, the sales peaked in the mornings. As the reader can see, this structure is highly reminiscent of the Cantor type construction from the previous paragraph. While this type of a phenomenon has been extensively studied in terms of seasonality, we believe that a fractal perspective can be of considerable value in view of the fact that if the specific months, weeks, days, and times in the hypothetical above change, while their relative number remains roughly the same, the seasonality considerations no longer apply, while the fractal dimension analysis, as we shall see, is still valid and effective. 

Another manifestation of the fractional dimensional phenomena in large data sets comes from the stock price data (see e.g. \cite{M2008}). It has been noted by several authors that fractional Brownian motion can be used to model stock price volatility. The fractional Brownian path has upper Minkowski dimension $>1$, reflecting the volatility of the data (see also \cite{AL21}). Combined with the discussion in the previous paragraph about the variants of seasonality, we arrive at a very interesting situation where we have a set of effective dimension $<1$ on the time axis, combined with the volatile data modeled by a function whose graph has dimension $>1$. A proper understanding of a situation of this type calls for advanced tools and perspectives from geometric measure theory, frame theory and harmonic analysis. 

The fractal phenomenon in data sets has ben studied before. We have been particularly influenced by the investigations by Smaller Jr., Turcotte and others in \cite{S87}, \cite{T86} and \cite{T87}. 

\vskip.125in 

\subsection{Structure of the paper} In order to describe our results, we need to set up the notion of fractional dimension of finite points. In Section \ref{subsectionbasics}, we develop the notion of Hausdorff dimension for families of finite point sets living in the $d$-dimensional unit cube ${[0,1]}^d$, $d \ge 2$. We place particular focus on point sets that are given as graphs of a function from ${[0,1]}^{d-1}$ to $[0,1]$, with the idea of modeling real life data sets where output, such as sales figures, is viewed as a function of the inputs that may include, the date, location, inflation rate, and other figures. 

\vskip.125in 

{\bf Theoretical results:} After developing the notion of Hausdorff dimension of point sets by introducing the discrete $s$-energy 
$$I_s(P_n)=n^{-2} \sum_{p \not=p'; p,p' \in P_n; |P_n|=n} {|p-p'|}^{-s},$$ we show (Theorem \ref{ksurface}) that if the discrete $s$-energy of a point set $P_n$ is suitably small, then $P_n$ cannot contain too many points on $k$-dimensional surfaces, $k<s$, thus showing that small $s$ energy limits the extent to which an effective dimension reduction can be implemented on a given data set. Note that $|p-p'|$ above denotes the Euclidean distance between $p$ and $p'$, whereas $|P_n|$ denotes the size of the finite set $P_n$. 

\vskip.125in 

{\bf Numerical experiments:} In order to illustrate the utility of the ideas described above, we apply the standard PCA (Principal Component Analysis) to a discretized version of the Cartesian product of two Cantor sets and show that PCA does not fundamentally distinguish between this sparse example and a scale integer grid. We then show that the discrete dimensionality of the same set is quickly and efficiently estimated using a Python implementation of the discrete energy function described above. These results are described in Subsection \ref{computationalshit} below and carried later in the paper. 

\vskip.125in 

{\bf Future directions:} After identifying discrete energy as an effective tool for determining the concentration of points of a numerical data set, we shall turn our attention towards understanding how the fractal nature of a data set can be exploited to make accurate forecasts using neural network models. More precisely, we shall ask how the architecture of a neural network should be influenced by the complexity of the data set as measure by discrete energy and other analytic tools. These ideas will be explored thoroughly in a sequel. 

\vskip.125in 

\section{Discrete fractional dimension and concentration of data sets} 
\label{subsectionbasics} 

\subsection{Definitions and core results}

In this subsection we develop some basic definitions of fractional dimension in a discrete setting and state some results that will be proven later in the paper.  

Let $P_n \subset {[0,1]}^d$, $d \ge 2$ be a finite set of size $n$. We consider a family of such sets, 
$$
   {\mathcal P}={\{P_n\}},
$$
where $n$ ranges over some subset of the positive integers. For $s \in [0,d]$, define the \emph{discrete $s$-energy} of $P_n$ as
\begin{equation} \label{discreteenergy}
    I_s(P_n)=n^{-2} \sum_{p \not=p'; p,p' \in P_n} {|p-p'|}^{-s},
\end{equation}
where here the sum is over pairs $(p,p')$ with $p,p' \in P_n$ and $p \neq p'$.

\vskip.125in 

\begin{definition} Let ${\mathcal P}$ be as above. Define the discrete Hausdorff dimension of ${\mathcal P}$ as 
$$ dim_{{\mathcal H}}({\mathcal P})=
\sup \left\{s \in [0,d]: \sup_{n} I_s(P_n)<\infty \right\}.$$
\end{definition} 

\vskip.125in 

We shall sometimes restrict our attention to point sets of the form 
\begin{equation} \label{function} G_n(f)=\left\{ (j/q, f(j/q)): j \in \mathbb Z^{d-1} \cap [0,q)^{d-1} \right\}, \qquad n = q^{d-1}, \end{equation}
for some $f:{[0,1]}^{d-1} \to [0,1]$. We shall refer to these as point set graphs. However, whenever possible, we shall establish results for the more general point sets defined above. 

Before we start stating and proving our results, we describe the practical motivation for the point sets considered above. Suppose that that a company $X$ has $d$ branches around the country, and we wish to describe the sales total for each branch on each of $n$ dates. The resulting point set is a 2-dimensional array  
$$\{p_{k,i}\}_{1 \leq k \leq n; 1 \leq i \leq d}.$$ 
We can regard this point set in two equivalent ways. We can fix a branch, labeled by $i$, and consider the time series array 
\begin{equation} \label{timeseries} \{p_{1i}, p_{2i}, \dots, p_{ni}\}. \end{equation} 
We have $d$ such time series, so in this way we may regard our point set as $d$ vectors in ${\mathbb R}^n$. Alternatively, we may fix a date, and consider a vector of sales figures at the different branches on a given day: 
\begin{equation} \label{fixedstore} \{p_{k1}, p_{k2}, \dots, p_{kd}\}. \end{equation} 
In this way, we may regard the point set as a collection of $n$ vectors in ${\mathbb R}^d$, as we have it set up above. 

Each of the two perspectives described above has its practical advantages from the point of view of results described below. We are going to show that if the discrete $s$ energy of a point set is suitably small, then the point set cannot be too concentrated on a lower dimensional hyper-plane determining a certain linear relationship between the data points. If we look at this from the point of view of the time series in (\ref{timeseries}), the linear relationship signifies the possibility that sales figures for different branches are strongly related to one another across the board in an easy to describe fashion. From the point of view of (\ref{fixedstore}), a concentration of data on a hyper-plane would say that for a significant number of different dates, there is a simple linear relationship between the sales figures for the various branches. 

Since the number of dates is likely to be considerably larger than the number of branches of a company, the analysis in each case is of a different nature, as we shall see below. 

We now turn our attention to the technical results. We begin with considering $n$ points in $d$-dimensional space, as we set up in the beginning of the section, but we shall described the ``flipped" perspective afterward. 

\vskip.25in 

\begin{lemma} \label{lemmaatleast1d} Let 
$\mathcal P = \{G_n\}$ be a family of point set graphs as in \eqref{function} above. Then $\dim_{{\mathcal H}}({\mathcal P}) \ge d-1$.
\end{lemma} 

\vskip.125in 


We want to be able to describe the support of a family $\mathcal P = \{P_n\}$ of point sets of a given discrete Hausdorff dimension. To do so, we introduce a little structure. Take a compact set $E \subset {\Bbb R}^d$ and a positive number $\delta \leq 1$. We let $N_E(\delta)$ denote the minimal number of closed balls of radius $\delta$ needed to cover $E$. Recall
$$ \underline{\dim}_{{\mathcal M}}(E)=\liminf_{\delta \to 0} \frac{\log(N_E(\delta))}{\log(1/\delta)} \qquad \text{ and } \qquad  \overline{\dim}_{{\mathcal M}}(E)=\limsup_{\delta \to 0} \frac{\log(N_E(\delta))}{\log(1/\delta)}$$ 
are the lower and upper Minkowski dimension of $E$, respectively. When the lower and the upper Minkowski dimensions agree, we call their common value the Minkowski dimension of $E$.

We will typically assume that, for some dimension $k$ (which need not be an integer),
\begin{equation}\label{finiteupperminkowski}
	N_E(\delta) \leq \max(C_E \delta^{-k}, 1) \qquad \text{ for } \delta > 0 
\end{equation}
for some constant $C_E$. Here, $C_E$ depends on $E$, but also on $k$ as well. If such a constant exists, we must have $\overline \dim_{\mathcal M}(E) \leq k$. Not only this, but $E$ must have finite $k$-dimensional upper Minkowski content. The condition \eqref{finiteupperminkowski} includes a wide range of essential examples, such as planar regions, piecewise-smooth manifolds, and a range of fractal sets including all of the Cantor sets we discuss in later sections.

We are ready for our core statement about the relationship between the discrete Hausdorff dimension and the Minkowski dimension of the set on which it is supported.

\begin{theorem} \label{theoremminimization} Let $P_n$ be a point set of size $n$ contained in a compact subset $E \subset \R^d$ satisfying \eqref{finiteupperminkowski} with $k > 0$. If $s > k$, we have the lower bound
\begin{equation}\label{quantitative lower bound}
    I_s(P_n) \geq \frac{k}{s - k} C_E^{-\frac s k} (n^{\frac s k - 1} - 1).
\end{equation}
Consequently, if $\mathcal P = \{P_n\}$ is a family of point sets contained in $E$ of upper Minkowski dimension $\overline \dim_{\mathcal M}(E) \leq k$, then $\dim_{\mathcal H} \mathcal P \leq k$.
\end{theorem}

\vskip.125in 

We now turn our attention to the question of dimension reduction. Suppose that $P_n$ is as above and $\dim_{{\mathcal H}}({\mathcal P})=s$. The question we ask is whether it is possible that the points of $\mathcal P$ concentrate on $k$-dimensional hyper-planes, or, more generally, on smooth $k$-dimensional surfaces. Our next result says that the fractal dimension places significant limitations on such possibilities. 

\begin{theorem} \label{ksurface} Let $E$ be a compact subset of $\R^d$ satisfying \eqref{finiteupperminkowski}, and let $P_n$ be a point set in $\R^d$ of size $n$. Then, for $s > k$,
\begin{equation} \label{dimensionreductionestimate} 
    |P_n \cap E| \leq \left(1 + C_E^\frac s k \left(\frac s k - 1 \right) I_s(P_n) \right)^\frac{1}{{\frac s k + 1}} n^\frac{2}{{\frac s k + 1}}.
\end{equation} 
Moreover, the exponent $\frac{2}{1 + \frac s k}$ is, in general, best possible.
\end{theorem}

\vskip.125in

\begin{remark} \label{pairflexibility} Since $P_n$ is finite, so is $I_s(P_n)$. Thus, in view of Theorem \ref{ksurface}, it is imperative to minimize the quantity
$$\left(1 + C_E^\frac s k \left(\frac s k - 1 \right) I_s(P_n) \right)^\frac{1}{{\frac s k + 1}} n^\frac{2}{{\frac s k + 1}}.$$
To see the interplay between the various quantities above, we assume first that $P_n$ has a diameter not exceeding $1$. Then, observe that for a fixed $n$, $I_s(P_n)$ is an increasing function of $s$, as is 
$$ {\left(\frac{s-k}{k} \right)}^{\frac{1}{1+\frac{s}{k}}}.$$
On the other hand, 
$n^{\frac{2}{1 + \frac{s}{k}}}$ is a decreasing function of $s$ (with $k$ fixed). Ultimately, an effective algorithm is needed to choose $s$ given $k$ and the point set $P_n$. 
\end{remark}

\vskip.125in

It is important, in practice, that we introduce a bit of ``wiggle room" into Theorems \ref{theoremminimization} and \ref{ksurface}, and instead consider point sets which are in some $\epsilon$-neighborhood of $E$. Thankfully, both theorems have thickened versions with the same exponenets. This flexibility is quite important for potential applications.

\begin{theorem}\label{thickenedminimization}
	Let $E$ be a compact subset of $\R^d$ satisfying \eqref{finiteupperminkowski} with $k > 0$. Let $P_n$ be a point set of size $n$ such that each point is within a distance of $\epsilon > 0$ from $E$. Then, if $s > k$, we have the discrete energy bound
\[
		I_s(P_n) \geq \frac{k}{s - k} C_E^{-\frac s k} \left(\left(1 - \frac{\epsilon A_{s,k}}{C_E^{\frac 1 k} n^{- \frac 1 k}} \right) n^{\frac s k - 1} - 1\right)
\]
with constant
\[
	A_{s,k} = \frac{s(s + 1)(s - k)}{k(s - k + 1)}.
\]
In particular
\[
	I_s(P_n) \geq \frac{k}{s - k} C_E^{-\frac s k} \left(\frac 1 2 n^{\frac s k - 1} - 1\right) \qquad \text{ if } \epsilon \leq \frac{C_E^{\frac 1 k}}{2 A_{s,k}} n^{-\frac 1 k}.
\]
\end{theorem}



\begin{theorem}\label{thickenedksurface} Let $P_n$ be a point set in $\R^d$ and let $E$ be a compact subset of $\R^d$ satisfying \eqref{finiteupperminkowski}. Then, if $E^\epsilon$ is the $\epsilon$-thickening of $E$ with
\[
	\epsilon = \frac{C_E^{\frac 1 k}}{2 A_{s,k}} n^{-\frac 1 k},
\]
then,
\[ 
	|P_n \cap E^\epsilon| \leq 2^\frac{1}{\frac s k + 1} \left(1 + C_E^\frac s k \left(\frac s k - 1 \right) I_s(P_n) \right)^\frac{1}{\frac s k + 1} n^\frac{2}{\frac s k + 1}.
\]
\end{theorem}

%

\vskip.125in

We will turn our attention to the case where our family of point sets lies along the graph of a function as in \eqref{function}. To apply the results from above, we will need to connect the regularity of the function to the dimension of its graph. Recall, we say that $f : [0,1]^{d-1} \to {\mathbb R}$ is H\"older continuous of order $\alpha \in (0,1]$ if there exists a constant $\rho>0$ such that
$$
	|f(x)-f(y)| \leq \rho {|x-y|}^{\alpha}
$$
for all $x,y \in [0,1]^{d-1}$.

\vskip.125in 

\begin{lemma} \label{biholderlemma} Let $G_n$ be as in \eqref{function}, and consider the corresponding ${\mathcal P}$. Suppose that $f: {[0,1]}^{d-1} \to [0,1]$ is H\"older continuous of order $\alpha \in (0,1]$. Then 
$$
	\dim_{{\mathcal H}}({\mathcal P}) \leq d - \alpha.
$$
\end{lemma}


Our last result of this section shows that every possible discrete dimension in $[d-1,d]$ is possible for the point set $G_n(f)$. 

\begin{theorem} \label{koch} For each $s \in [d-1,d]$ there exists $f: {[0,1]}^{d-1} \to [0,1]$,  such that if ${\mathcal P}=\{G_n(f)\}$, (with $G_n(f)$ defined as above), then $dim_{{\mathcal H}}({\mathcal P})=s$. \end{theorem} 

\vskip.125in

\subsection{Computational results and examples} 
\label{computationalshit} 

\vskip.125in 

\subsubsection{Discrete dimension of Cartesian products of Cantor-type sets} Start with the interval $[0,1]$ and positive integers $m,n\in\mathbb Z^+$ s.t. $m<n$. Divide $[0,1]$ into $n$ subintervals of equal length, and choose $m$ of those intervals. Let $C_{m,n}^1$ be the set of endpoints of the intervals chosen. At the $k$-th step, split each of the remaining intervals into $n$ equal subintervals, choose $m$ of those intervals, and let $C_{m,n}^k$ be the set of endpoints of the intervals chosen. At each step, we pick the same $m$ subintervals from the remaining intervals. Since the $m$ subintervals we choose to keep are arbitrary, the set $C_{m,n}^k$ is not unique and merely denotes one set satisfying these conditions.

\begin{theorem} \label{conor} 
Let $d\in\mathbb Z^+$. Let $C_{m_1,n_1}^k,\dots,C_{m_d,n_d}^k$ be discrete Cantor sets. Set $A_k=\prod_{i=1}^dC_{m_i,n_i}^k$. Then 
$$dim_{{\mathcal H}}(A_k)=\frac{\ln(m_1)}{\ln(n_1)}+\cdots+\frac{\ln(m_d)}{\ln(n_d)}.$$
\end{theorem}

\vskip.125in 

Now that we have this method computing the dimension of discrete Cantor set products, we can compute the discrete $s$-energy of the sets in these families to better observe their behavior. For instance, if we compute the discrete $s$-energy of the discrete Cantor set products from Figure \ref{productcantorpicture} below, where for each set, $s$ is the dimension of that set computed by Theorem \ref{conor}, we can see in Figure \ref{fig4}, Figure \ref{fig5}, and Figure \ref{fig6} that the discrete $s$-energy of the sets in each case increases slowly. This is what we would expect given that for each family of sets, $s$ is the critical value, and for any $s'<s$, the discrete $s'$-energy of the sets are all bounded by a single constant.
\begin{figure}
    \centering
    \includegraphics[scale=.5]{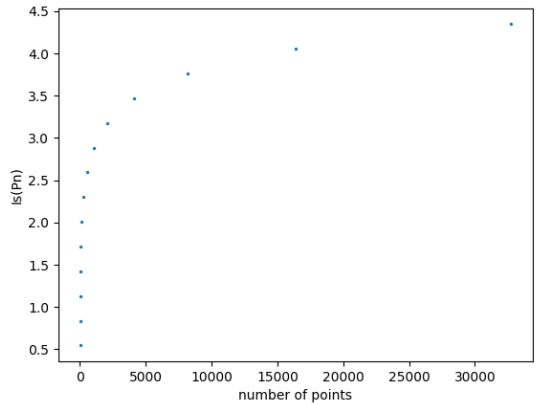}
    \caption{Discrete $s$-energy of $C_{2,4}^k$ for $1\leq k\leq 15$ with $s=\frac{1}{2}$.}
    \label{fig4}
\end{figure}

Meanwhile, if we compute the discrete 2-energy of a family of lattices in $[0,1]^2$, then we can see from Figure \ref{fig7} that the discrete 2-energy of the lattices increases slowly as well, possibly suggesting that the family of lattices is 2-adaptable. This would make sense intuitively because we would expect the family of lattices to be 2-dimensional, suggesting that this new notion of dimension may align with our general notion of dimension.

\vskip.125in 

\subsubsection{PCA analysis of fractal and non-fractal objects} 
\label{pcasubsection} 

As we noted in the introduction, the classical PCA method does not distinguish some key types of data sets. For example, PCA is not great at distinguishing a grid from, say, a Cartesian product of two Cantor sets as we shall discuss below. 

Recall the process of PCA: suppose there is a data set $X=\{x_1,x_2,x_3,....x_n\}, x_i\in R^D$ and we want to reduce its dimension to a lower one without losing too much of the information. Firstly, the data set needs to be centralized for an easier covariance matrix calculation later, that is, for each component $x_p$ ($1\le p \le D$) in $R^D$, we subtract the coordinate of each point in component $x_p$ by the average of the sum of all points' corresponding coordinates: $x_\text{q,p} \to x_\text{q,p}-\frac{1}{n}*\sum\limits_{i=1}^n{x_\text{q,p}}$.
Then we find the data covariance matrix $XX^T$, calculating the m-th largest eigenvalues of that covariance matrix, and thus we can get the m corresponding eigenvectors $w_1,w_2,...w_m$, named as m principle components. After generating a new matrix $W=\{w_1,w_2,w_3,....w_m\}, w_i\in R^D$, the compressed data becomes $y_i=W^Tx_i$ for each $x_i$. We reduce the dimension of data set to $X'=\{y_1,y_2,y_3,....y_n\}, y_i\in R^M$.

To illustrate some limitations of PCA, we test it with the grid and Cartesian product of two Cantor sets. By doing the PCA analysis as described above, the covariance matrix of both the grid and Cartesian product of two Cantor sets are diagonal and thus the eigenvalues for both data distributions are identical. Such results indicate that the two corresponding eigenvectors explains the same weight of the variance of the data set, while reducing the dimension under such condition causes the lost of too much information. Therefore, PCA regards both the 2-D Lattice and Cartesian product of two Cantor sets as 2 dimensional sets.

Nevertheless, the discrete-s energy remains robust to identify the dimension of a cantor set is, in fact, less than 1. For instance, our codes show that the Discrete $s$-energy of $C_{2,4}^k$ converges to a relative small number as the number of points(depends on k) increases, given $ s<\frac{ln(2)}{ln(4)}$. The results are consistent with Theorem \ref{conor}. 


\begin{figure}
    \centering
    \includegraphics[scale=.5]{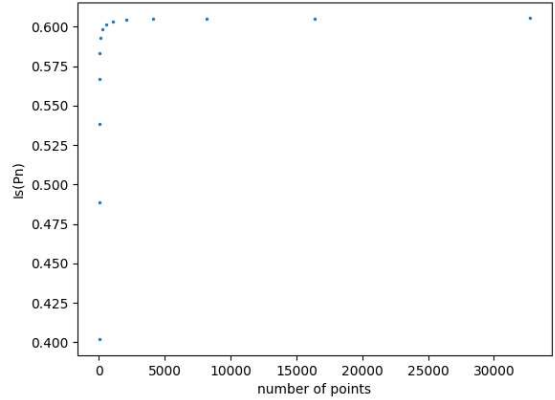}
    \caption{Discrete s-energy of $C_{2,4}^k$ with $2^\text{k+1}$ points for $1\leq k\leq 15$ with s=0.1}
    \label{fig6}
\end{figure}

\begin{figure}
    \centering
    \includegraphics[scale=.5]{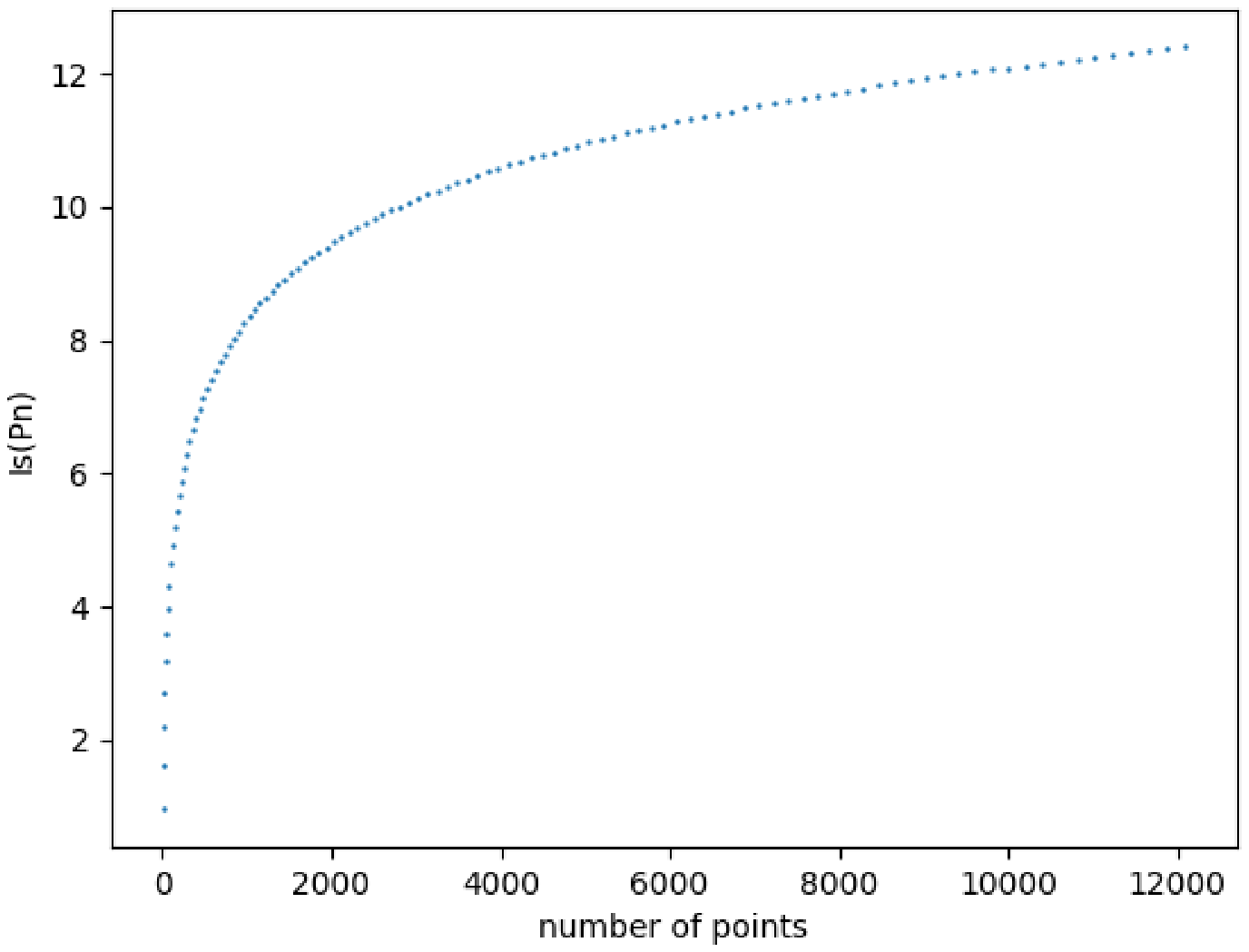}
    \caption{Discrete 2-energy of 2-dimensional lattice with $(k+1)^2$ points for $1\leq k\leq 110$, $s=2$}
    \label{fig7}
\end{figure}

\begin{figure}
    \centering
    \includegraphics[scale=.5]{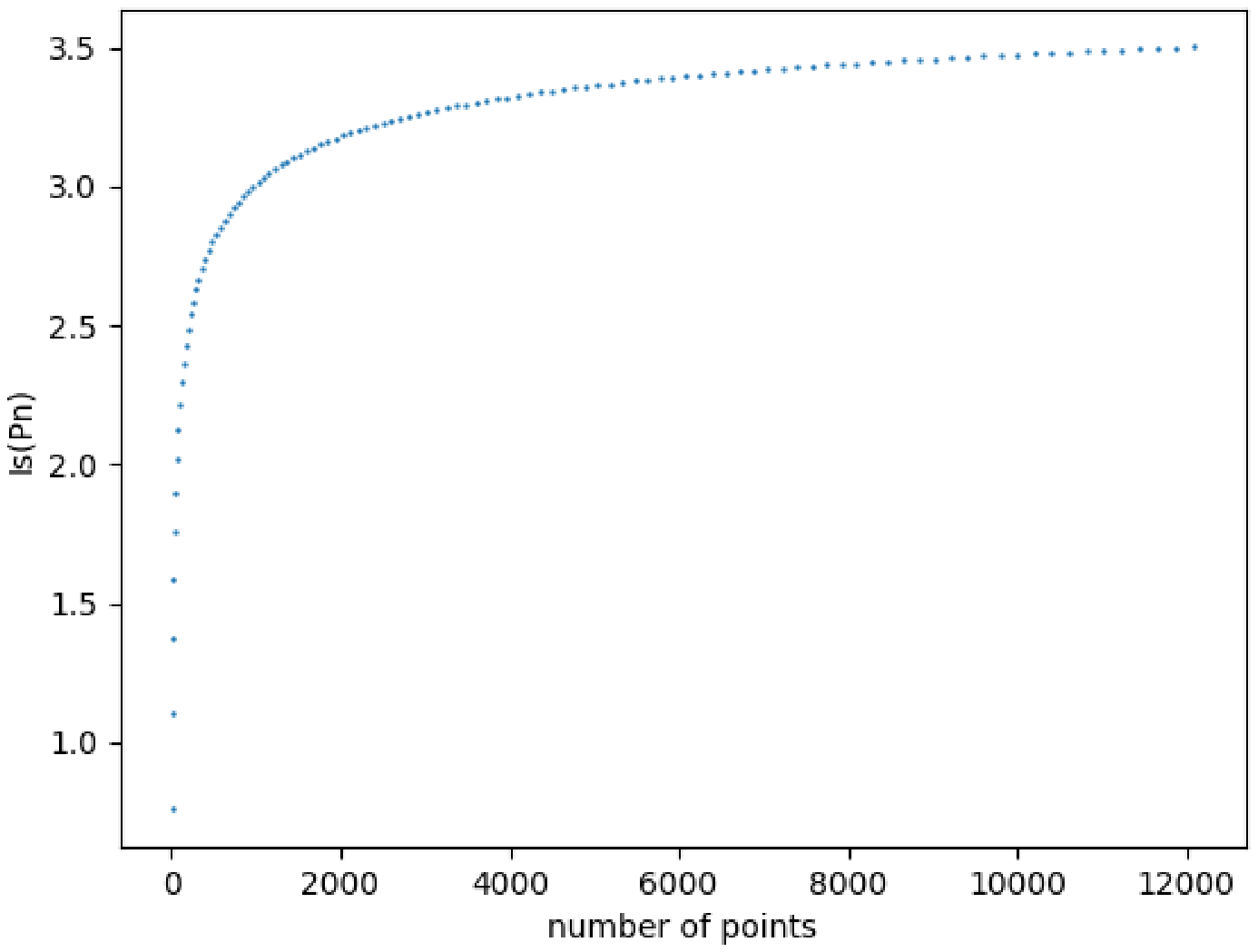}
    \caption{Discrete s-energy of 2-dimensional lattice with $(k+1)^2$ points for $1\leq k\leq 110$, $s=1.5$}
    \label{fig7'}
\end{figure}

\vskip.125in 

\begin{figure}
\centering
\includegraphics[scale=.4]{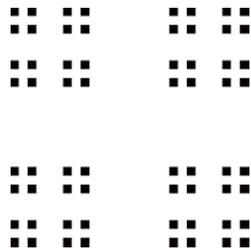}
\caption{The Cartesian products of two middle-half Cantor sets (otherwise known as the Garnett set)} 
\label{productcantorpicture} 
\end{figure}

\vskip.25in 

\section{Proof of results from Section \ref{subsectionbasics}}

\vskip.125in 

\subsection{Proof of Lemma \ref{lemmaatleast1d}} We will show the $s$-discrete energy of $G_n$ is uniformly bounded if $0 \leq s < d-1$. As usual, we let $n = q^{d-1}$ for a positive integer $q$. Let $p = (j/q,f(j/q)) \in G_n$, where here $j \in \mathbb Z^{d-1} \cap [0,q)^{d-1}$. Note, for each such pair $p,p'$,
\[
    |p - p'|^{-s} \leq |j/q - j'/q|^{-s} = q^{s} |j - j'|^{-s}.
\]
Hence, we have
\[
    I_s(P_n) = n^{-2} \sum_{p \neq p'} |p - p'|^{-s} \leq q^{-2(d-1) + s} \sum_{j \neq j'} |j - j'|^{-s},
\]
where here $j$ and $j'$ range over $\mathbb Z^{d-1} \cap [0,q)^{d-1}$. By a change of variables, the above is bounded by
\[
    q^{-(d-1) + s} \sum_{\substack{j \in \Z^{d-1} \cap (-q,q)^{d-1} \\ j \neq 0}} |j|^{-s} \leq C q^{-(d-1) + s} q^{d-1 - s}
\]
by the integral test, where here $C$ is a constant depending on $d$ and $s$ only. The conclusion of Lemma $\ref{lemmaatleast1d}$ follows.


\vskip.125in 

\subsection{Proofs of Theorems \ref{theoremminimization} and \ref{thickenedminimization}}

We start by only assuming that $E$ is a compact subset of $[0,1]^d$ and $P_n$ is any set of $n$ points in $E$. We will not introduce the regularity condition \eqref{finiteupperminkowski} until a bit later.

We write
\[
    |p - p'|^{-s} = s \int_0^\infty \mathbf 1_{[0,\infty)}(r - |p - p'|) r^{-s - 1} \, dr,
\]
so that
\begin{align*}
    I_s(P_n) &= n^{-2} \sum_{p \neq p'} |p - p'|^{-s} \\
    &= s n^{-2} \sum_{p \neq p'} \int_0^\infty \mathbf 1_{[0,\infty)}(r - |p - p'|) r^{-s - 1} \, dr \\
    &= s n^{-2} \int_0^\infty \left( |\{ (p,p') : |p - p'| \leq r\}| - n  \right) r^{-s-1} \, dr.
\end{align*}
The subtracted $n$ is here to eliminate the contribution to $|\{ (p,p') : |p - p'| \leq r\}|$ of the pairs $(p,p')$ with $p = p'$. To estimate $I_s(P_n)$, we will need a lower bound on the size of this set.

\begin{lemma} Let $E$ be a compact subset of $\R^d$ and let $P_n \subset E$ be a point set in $E$ of size $n$. Then,
\[
    |\{ (p,p') \in P_n \times P_n : |p - p'| \leq r \}| \geq \frac{n^2}{N_E(r)}.
\]
\end{lemma}

\begin{proof}
    Take a minimal cover of $E$ by $N_E(r)$ balls of radius $r$. Select a partition of $E$ into $N_E(r)$ pairwise disjoint sets
    \[
        E = E_1 \cup \cdots \cup E_{N_E(r)}
    \]
    where each part $E_i$ is contained in its respective ball in the cover. Note,
    \[
        |\{(p,p') : |p - p'| \leq r \}| \geq \sum_j |E_j \cap P_n|^2 \geq \frac{1}{N_E(r)} \left( \sum_j |E_j \cap P_n| \right)^2 = \frac{n^2}{N_E(r)},
    \]
    where the second inequality is an application of Cauchy-Schwarz. This concludes the proof of the lemma.
\end{proof}

\begin{remark}
    The Cauchy-Schwarz inequality used above is optimal when $|E_j \cap P_n|$ is constant across $j$, i.e. when $P_n$ is well-distributed in $E$. If, say, $P_n$ is contained in a plane $E$ in $[0,1]^d$, then the energy $I_s(P_n)$ is minimized (or nearly minimized) if $P_n$ is a lattice in $E$.
\end{remark}

We now resume the proof of Theorem \ref{theoremminimization}. Taking $\delta$ to be any number with $0 < \delta \leq 1$. Then, the lemma gives
\[
    |\{(p,p') : |p - p'| \leq r \}| - n \geq
    \begin{cases}
        0 & r < \delta \\
        N_E(r)^{-1} n^2 - n & r \geq \delta.
    \end{cases}
\]
It now follows that
\[
	I_s(P_n) \geq sn^{-2} \int_\delta^\infty \left( N_E(r)^{-1}n^2 - n \right) r^{-s-1} \, dr.
\]
After some minor calculations, it follows:

\begin{lemma}\label{energy lower bound lemma}
	Let $E$ be a compact subset of $[0,1]^d$ and $P_n$ a set of $n$ points in $E$. Then,
	\[
		I_s(P_n) \geq s \left( \int_\delta^\infty N_E(r)^{-1} r^{-s-1} \, dr \right) - n^{-1} \delta^{-s}
	\]
	where $0 < \delta \leq 1$.
\end{lemma}

We are now ready to prove Theorem \ref{theoremminimization}.

\begin{proof}[Proof of Theorem \ref{theoremminimization}]
	We now can assume \eqref{finiteupperminkowski} from which we have
	\[
		N_E(r)^{-1} \geq \min(C_E^{-1} r^k, 1).
	\]
	Hence by the lemma, we have
	\begin{align*}
		I_s(P_n) &\geq C_E^{-1} s \int_\delta^{C_E^{\frac 1 k}} r^{k - s - 1} \, dr + s \int_{C_E^\frac 1 k}^\infty r^{-s - 1} \, dr - n^{-1} \delta^{-s} \\
		&= \frac{C_E^{-1} s}{k - s} \left( C_E^{1 - \frac s k} - \delta^{k - s} \right) + C_E^{-\frac s k} - n^{-1} \delta^{-s}.
	\end{align*}
	Taking $\delta = (C_E^{-1} n)^{-\frac 1 k}$ yields the bound, after some simplification.
\end{proof}


Theorem \ref{thickenedminimization} is also now within easy reach.

\begin{proof}[Proof of Theorem \ref{thickenedminimization}] Let $E^\epsilon$ denote the $\epsilon$-thickening of $E$. Now, if $E$ satisfies \eqref{finiteupperminkowski}, then
\[
	N_{E^\epsilon}(r) \leq N_E(r - \epsilon) \leq \max(C_E(r - \epsilon)^{-k},1).
\]
By the lemma, we have
\begin{align*}
	I_s(P_n) &\geq s\int_{\delta + \epsilon}^\infty (N_E(r - \epsilon))^{-1} r^{-s - 1} \, dr - n^{-1} (\delta + \epsilon)^{-s} \\
	&\geq s \int_\delta^\infty N_E(r)^{-1} (r + \epsilon)^{-s - 1} \, dr - n^{-1} \delta^{-s}.
\end{align*}
Since $r^{-s-1}$ is convex, we have
\[
	(r + \epsilon)^{-s - 1} \geq r^{-s - 1} - (s + 1) \epsilon r^{-s - 2} \qquad \text{ for } r > 0.
\]
Hence,
\[
	I_s(P_n) \geq \left( s \int_\delta^\infty N_E(r)^{-1} r^{-s - 1} \, dr - n^{-1} \delta^{-s} \right) - s(s + 1) \epsilon \int_{\delta}^\infty N_E(r)^{-1} r^{-s - 2} \, dr.
\]
Taking $\delta = C_E^{\frac 1 k} n^{-\frac 1 k}$ and repeating the argument in the proof of Theorem \ref{theoremminimization} yields
\begin{align*}
	I_s(P_n) &\geq \frac{k}{s - k} C_E^{-\frac s k} (n^{\frac s k - 1} - 1) - \epsilon s (s + 1) \left(  C_E^{-1} \int_{C_E^{\frac 1 k} n^{-\frac 1 k}}^{C_E^\frac 1 k} r^{k - s - 2} \, dr + \int_{C_E^\frac 1 k}^\infty r^{-s-2} \, dr \right).
\end{align*}
By change of variables, we have
\[
	C_E^{-1} \int_{C_E^{\frac 1 k} n^{-\frac 1 k}}^{C_E^{\frac 1 k}} r^{k - s - 2} \, dr 
	= \frac{C_E^{-\frac{s+1}{k}}}{s - k + 1} \left( n^{-1 + \frac{s + 1}{k}} - 1 \right)
\]
and
\[
	\int_{C_E^\frac 1 k}^\infty r^{-s-2} \, dr = \frac{C_E^{- \frac{s + 1}{k}}}{s + 1}.
\]
Their sum is then bounded above by $\frac{C_E^{-\frac{s+1}{k}}}{s - k + 1}  n^{-1 + \frac{s + 1}{k}}$, and hence we have the lower bound
\[
	I_s(P_n) \geq \frac{k}{s - k} C_E^{-\frac s k} (n^{\frac s k - 1} - 1) - \frac{\epsilon}{C_E^\frac 1 k n^{-\frac 1 k}} \frac{C_E^{-\frac{s}{k}} s(s+1)}{s - k + 1}  n^{\frac{s}{k} - 1}.
\]
The bound in the theorem follows.
\end{proof}

\vskip.125in 

\subsection{Proofs of Theorems \ref{ksurface} and \ref{thickenedksurface}}


\begin{proof}[Proof of Theorem \ref{ksurface}] For fixed $n$, let $m = |P_n \cap E|$ and $P_m' = P_n \cap E$. On one hand, we have an upper bound
\[
    I_s(P_m') = m^{-2} \sum_{\substack{p \neq p' \\ p,p' \in P_{m'}}} |p-p'|^{-s} \leq m^{-2} \sum_{p \neq p'; p,p' \in P_n} |p-p'|^{-s} = m^{-2} n^2 I_s(P_n)
\]
on $I_s(P_m')$. At the same time, since $P_m' \subset E$, \eqref{quantitative lower bound} gives us a lower bound
\[
	I_s(P_m') \geq \frac{k}{s - k} C_E^{-\frac s k} (m^{\frac s k - 1} - 1).
\]
By comparing the upper and lower bounds, we obtain
\[
	(m^{\frac s k - 1} - 1)m^2 \leq C_E^\frac s k \left(\frac s k - 1 \right) I_s(P_n) n^2.
\]
The left side is bounded below by $m^{\frac s k + 1} - n^2$, and hence we have
\[
	m^{\frac s k + 1} \leq \left(1 + C_E^\frac s k \left(\frac s k - 1 \right) I_s(P_n) \right) n^2.
\]
The bound in the theorem follows.


\vskip.125in 

Theorem \ref{ksurface} is, in general, best possible. To see this, we start with a family $\mathcal P$ of lattices
$$
	P_n=\left\{ \frac{j}{n^{\frac{1}{d}}}: j \in {\Bbb Z}^d, 0 \leq j_i \leq n^{\frac{1}{d}}, 1 \leq i \leq d \right\}.
$$ 
It is not difficult to check that $I_s(P_n) \leq C$ independently of $n$ if $s<d$, from which we have $\dim_{\mathcal H} \mathcal P = d$.
Now let us modify $P_n$ by replacing the $n^{\frac{k}{d}}$ points in the plane $x_1 = \cdots = x_k = 0$ with a lattice $Z := m^{-1/k} \Z^k \cap [0,m^{1/k})^k$ of $m$ points. We will take $m \gg n^\frac{k}{d}$, which effectively increases the number of points on the plane. However, we will not add so many as to disturb the bounds on the $s$-energy of $\mathcal P$. To this end, we require that if $s < k$, there is a constant $C$ for which
\[
	I_s(P_n) \leq C \qquad \text{ for each $n$}.
\]
The left side reads as a bounded term plus
$$
	n^{-2} \sum_{\substack{u,u' \in Z \\ u \neq u'}} {\left|\frac{u-u'}{m^{\frac{1}{k}}} \right|}^{-s} \approx n^{-2} m^{\frac{s}{k}+1} \sum_{w \in Z} {|w|}^{-s} \approx n^{-2} m^{\frac{s}{k}+1},
$$
and hence the desired bounds hold if and only if $m^{\frac{s}{k}+1} \leq Cn^2$, i.e $m \leq Cn^{\frac{2}{1+\frac{s}{k}}}$. Here and throughout, $X \approx Y$ means that there exists a uniform constant $C$ such that $C^{-1}Y \leq X \leq CY$. 
\end{proof}

Next, we proceed with the proof of Theorem \ref{thickenedksurface}, the thickened version of Theorem \ref{ksurface}. 

\begin{proof}[Proof of Theorem \ref{thickenedksurface}] We follow the proof of Theorem \ref{ksurface} above, except we take $P_m' = P_n \cap E^\epsilon$ where here
\[
	\epsilon = \frac{C_E^{\frac 1 k}}{2 A_{s,k}} n^{-\frac 1 k} \leq \frac{C_E^{\frac 1 k}}{2 A_{s,k}} m^{-\frac 1 k}
\]
with notation as in Theorem \ref{thickenedminimization}. The same theorem then tells us
\[
	I_s(P_m') \geq \frac{k}{s - k} C_E^{-\frac s k} \left(\frac 1 2 m^{\frac s k - 1} - 1\right).
\]
We similarly have an upper bound
\[
	I_s(P_m') \leq m^{-2} n^2 I_s(P_n).
\]
Proceeding as before, we obtain
\[
	\frac 1 2 m^{\frac s k + 1} - n^2 \leq \frac 1 2 m^{\frac s k + 1} - m^2 \leq C_E^\frac s k \left(\frac s k - 1 \right) I_s(P_n) n^2.
\]
It follows
\[
	m \leq 2^\frac{1}{\frac s k + 1} \left(1 + C_E^\frac s k \left(\frac s k - 1 \right) I_s(P_n) \right)^\frac{1}{\frac s k + 1} n^\frac{2}{\frac s k + 1}
\]
as needed.
\end{proof}

\vskip.25in 

\subsection{Proof of Lemma \ref{biholderlemma}}

We claim the graph of $f$ satisfies \eqref{finiteupperminkowski} for $k = d - \alpha$, afterwards Theorem \ref{theoremminimization} completes the proof. We do this instead by covering the graph of $f$ with cubes of length $\delta$. Passing from cubes to balls will only affect the constant in \eqref{finiteupperminkowski}.

For fixed $q$, let $Q_j$ denote the cube $\prod_{i=1}^{d-1}[\frac{j_i}{q},\frac{j_i+1}{q}]$ where here $j$ ranges over $\mathbb Z^{d-1} \cap [0,q)^{d-1}$. Since $f$ is H\"older continuous of order $\alpha$, $f(Q_j)$ has diameter at most $\rho (\sqrt{d-1}q^{-1})^\alpha$. Hence, we can cover the graph of $f$ over $Q_j$ by $\rho \sqrt{d-1}q^{1-\alpha}$ $d$-dimensional cubes of sidelength $q^{-1}$. Repeating for each of the $q^{d-1}$ cubes $Q_j$, we cover the whole graph of $f$ by $q^{d-1} \cdot \rho \sqrt{d-1} q^{1 - \alpha} = \rho \sqrt{d-1}q^{d - \alpha}$ cubes of length $q^{-1}$, and our claim is proved after taking $\delta = q^{-1}$.

\vskip.125in

\section{Proof of Theorem \ref{koch}} 

\subsection{2-dimensional case}

\begin{proof}
In order to prove the case for $d=2$, the construction used by B.Hunt is followed \cite{Hunt}. In such construction, we consider Weierstrass functions $f_{\theta}:{[0,1]} \to [0,1]$ with random phases of the form:
\begin{equation} \label{Weierstrass functions with random shift} f_{\theta}(x)=\sum_{n=0}^{\infty} a^{n}cos\left(2\pi(b^{n}x+\theta_{n})\right)
\end{equation}
where $0<a<1<b$, $ab<1$, and $\theta=(\theta_{1},\theta_{2},...) \in [0,1]^{\infty}=H$ is randomly chosen by sampling each of its entries according to the uniform distribution on $[0,1]$.
\begin{figure}[h!]
    \centering
    \includegraphics[width=6cm]{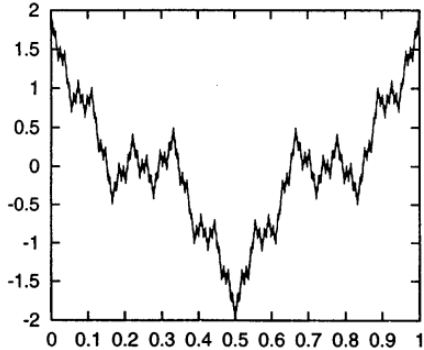}
    \caption{Graph of $f(x)$ with $a=0.5$ and $b=3$ and $\theta_n=0$ for any $n$.}
    \label{fig:wierstrass}
\end{figure}
\\
Note that, for any $\theta \in H$, $f_{\theta}$ is an Holder continuous function with exponent $\alpha = -\frac{\log(a)}{\log(b)}$. In fact, for any $x,y \in [0,1]$, substituting $a=b^{-\alpha}$:
\begin{equation}
\begin{split}
    \abs{f_{\theta}(x)-f_{\theta}(y)} &= \left|\sum_{n=0}^{\infty} b^{-n\alpha}\cos(2\pi(b^{n}x+\theta_{n}))-\sum_{n=0}^{\infty} b^{-n\alpha}\cos(2\pi(b^{n}y+\theta_{n})) \right| \\
    &\leq \sum_{n=0}^{\infty} b^{-n\alpha}\min(2,2\pi b^{n}\abs{x-y}) \\
    &\leq \sum_{n=0}^{m-1} 2\pi b^{n(1-\alpha)}\abs{x-y}+\sum_{n=m}^{\infty} 2b^{-n\alpha} \\
    &= 2\pi\frac{b^{(1-\alpha)m}-1}{b^{(1-\alpha)}-1}\abs{x-y}+2\frac{b^{-\alpha m}}{1-b^{-\alpha}}
\end{split}
\end{equation}
for any integer $m>0$. The first inequality followed thanks to the fact that the cosine function is Lipschitz (immediate consequence of the mean value theorem). Set $m$ to be the positive integer such that $b^{-m}<\abs{x-y}\leq b^{-(m-1)}$. Then, resuming the estimate:
\begin{equation} \label{alpha Holder exponent}
\begin{split}
\abs{f_{\theta}(x)-f_{\theta}(y)} &\leq 2\pi\frac{\left(\frac{b}{\abs{x-y}}\right)^{(1-\alpha)}}{b^{(1-\alpha)}-1}\abs{x-y}+2\frac{\abs{x-y}^{\alpha}}{1-b^{-\alpha}}\\
&\leq \left(\frac{2\pi b^{(1-\alpha)}}{b^{(1-\alpha)}-1}+\frac{2}{1-b^{-\alpha}}\right)\abs{x-y}^\alpha.
\end{split}
\end{equation}
Thus, combining \eqref{alpha Holder exponent} and Lemma $\ref{biholderlemma}$, we have that $dim_{{\mathcal H}}({\mathcal P})\geq 2-\alpha$. To prove the other inequality, the random phases are very useful as they allow us to estimate the energy integral indirectly. In fact, the key point of the proof is to show that, for any $s<2-\alpha$, $\EX_{H}(I_{s}(G_{n}))<\infty$ independently from $n$ (i.e. independently from $q$). Consequently, this implies that for almost every random sequence of phases $\theta \in H$, $f_{\theta}$ satisfies $dim_{{\mathcal H}}({\mathcal P})\leq 2-\alpha$ (completing the proof for $d=2$). To perform the estimate, by linearity of expectation:
\begin{equation}
\EX_{H}(I_{s}(G_{n}))=q^{-2}\sum_{\substack{j\neq j'\\j,j'\in [0,q)\cap \Z}}\EX_{H}\left(\abs{\left(\frac{j}{q},f_{\theta}\left(\frac{j}{q}\right)\right)-\left(\frac{j'}{q},f_{\theta}\left(\frac{j'}{q}\right)\right)}^{-s}\right)
\end{equation}
Next, the following lemma (proved by B.Hunt) provides an upper bound for the expectations inside the sum when $\frac{j}{q}$ and $\frac{j}{q}$ are close enough.
\begin{prop}\label{Proposition}
For any $x,y$ and any such that $\abs{x-y}<1/2b^{2}$:
\begin{equation}\label{lemma}
\EX_{H}\left(\abs{\left(x,f_{\theta}(x)\right)-\left(y,f_{\theta}(y)\right)}^{-s}\right)\leq C\abs{x-y}^{(1-\alpha-s)}
\end{equation}
for any $s\in(1,2-\alpha)$
\end{prop}

Then, resuming the proof, we can decompose the sum in two terms:
\begin{equation}
\EX_{H}(I_{s}(G_{n}))= I + II 
\end{equation}
where
\begin{equation}
I=q^{-2}\sum_{\substack{j\neq j'\\j,j'\in [0,q)\cap \Z\\\abs{j-j'}\leq \frac{q}{2b^{2}}}}\EX_{H}\left(\abs{\left(\frac{j}{q},f_{\theta}\left(\frac{j}{q}\right)\right)-\left(\frac{j'}{q},f_{\theta}\left(\frac{j'}{q}\right)\right)}^{-s}\right)
\end{equation}
\begin{equation}
II=q^{-2}\sum_{\substack{j\neq j'\\j,j'\in [0,q)\cap \Z\\\abs{j-j'}> \frac{q}{2b^{2}}}}\EX_{H}\left(\abs{\left(\frac{j}{q},f_{\theta}\left(\frac{j}{q}\right)\right)-\left(\frac{j'}{q},f_{\theta}\left(\frac{j'}{q}\right)\right)}^{-s}\right)
\end{equation}
The second term is easy to bound as:
\begin{equation}
II\leq q^{-2}\sum_{\substack{j\neq j'\\j,j'\in [0,q)\cap \Z\\\abs{j-j'}> \frac{q}{2b^{2}}}}\abs{\frac{j}{q}-\frac{j'}{q}}^{-s}\leq q^{-2}q^{2}(2b^2)^{s}=(2b^2)^{s}
\end{equation}
For the first term, using proposition \eqref{Proposition}:
\begin{equation}
\begin{split}
I &\lesssim q^{-2}\sum_{\substack{j\neq j'\\j,j'\in [0,q)\cap \Z\\\abs{j-j'}\leq \frac{q}{2b^{2}}}}\abs{\frac{j}{q}-\frac{j'}{q}}^{1-\alpha-s}\lesssim q^{-3+\alpha+s}\sum_{\substack{j\neq j'\\j,j'\in [0,q)\cap \Z}}\abs{j-j'}^{1-\alpha-s} \\
&\lesssim q^{-2+\alpha+s}\sum_{\substack{j\in (-q,q)\cap \Z\\j\neq0}}\abs{j}^{1-\alpha-s}\lesssim q^{-2+\alpha+s}q^{2-\alpha-s}\lesssim 1
\end{split}
\end{equation}
completing the proof. Note that the third and fourth inequality are, respectively, due to the change of variable $j-j'$ to $j$ and due the to integral test (where the integral converges for values of $s$ strictly smaller than $2-\alpha$, as desired). The 2-dimensional case is fully proved since for any given value $x\in(0,1]$ we can find values of $a$ and $b$ such that $\alpha=x$. This is a consequence of the fact that $\alpha=-\frac{\log(a)}{\log{b}}$ and $0<a<1<b$ with $ab\geq1$. In fact, fixing a positive $a<1$, $b=\frac{1}{a^{x}}\geq \frac{1}{a}$ for any value of $x\in (0,1]$ and $\alpha=x$.
\end{proof}

\subsection{Higher dimensional case for $d\geq3$}
\begin{proof}
To generalize to higher dimensions (i.e. $d\geq3$), we consider the following functions $g_{\theta}:{[0,1]^{d-1}} \to [0,1]$ of the form:
\begin{equation} \label{d-1 dimensional functions} g_{\theta}(x)=g_{\theta}(x_1,...,x_{d-1})=f_{\theta}(x_1)
\end{equation}
where the functions $f_{\theta}$ are defined as in \eqref{Weierstrass functions with random shift}.\\
For any $\theta\in H$, $g_{\theta}$ is again Holder continuous with exponent $\alpha$. This is because, using \eqref{alpha Holder exponent} and \eqref{d-1 dimensional functions}, for any $x,y \in [0,1]^{d-1}$:

\begin{equation}\label{alpha Holder exponent for g_theta}
\begin{split}
\abs{g_{\theta}(x)-g_{\theta}(y)}=\abs{f_{\theta}(x_1)-f_{\theta}(y_1)} & \leq \left(\frac{2\pi b^{\left(1-\alpha\right)}}{b^{\left(1-\alpha\right)}-1}+\frac{2}{1-b^{-\alpha}}\right)\abs{x_1-y_1}^\alpha\ \\
& \leq\left(\frac{2\pi b^{\left(1-\alpha\right)}}{b^{\left(1-\alpha\right)}-1}+\frac{2}{1-b^{-\alpha}}\right)\abs{x-y}^\alpha.
\end{split}
\end{equation}
Thus, combining \eqref{alpha Holder exponent for g_theta} and Lemma $\ref{biholderlemma}$, $dim_{{\mathcal H}}({\mathcal P})\geq d-\alpha$.

To prove the other inequality, looking at the expectation over all random phases as in the case $d=2$, we can perform the following decomposition:
\begin{equation}
\begin{split}
\EX_{H}(I_{s}(G_{n})) &=q^{-2(d-1)}\EX_{H}\left(\sum_{\substack{j\neq j'\\j,j'\in [0,q)^{d-1}\cap \Z^{d-1}}}\abs{\left(\frac{j}{q},g_{\theta}\left(\frac{j}{q}\right)\right)-\left(\frac{j'}{q},g_{\theta}\left(\frac{j'}{q}\right)\right)}^{-s}\right) \\
& =q^{-2(d-1)}\EX_{H}\left(\sum_{\substack{j\neq j'\\j,j'\in [0,q)^{d-1}\cap \Z^{d-1}}}\abs{\left(\frac{j}{q},f_{\theta}\left(\frac{j_{1}}{q}\right)\right)-\left(\frac{j'}{q},f_{\theta}\left(\frac{j'_{1}}{q}\right)\right)}^{-s}\right)\\
&=I + II
\end{split}
\end{equation}
where
\begin{equation}
\begin{split}
I=q^{-2(d-1)}\sum_{\substack{j_{1}\neq j'_{1}\\j_1,j'_1\in [0,q)\cap \Z}}\EX_{H}\left(\sum_{\bar{j},\bar{j'}\in [0,q)^{d-2}\cap \Z^{d-2}}\left(\abs{\frac{j_{1}}{q}-\frac{j'_{1}}{q}}^2+\abs{f_{\theta}\left(\frac{j_{1}}{q}\right)-f_{\theta}\left(\frac{j'_{1}}{q}\right)}^2+\abs{\frac{\bar{j}}{q}-\frac{\bar{j'}}{q}}^2\right)^{-\frac{s}{2}}\right)
\end{split}
\end{equation}
\begin{equation}
II=q^{-2(d-1)}\sum_{\substack{j_{1}=j'_{1}\\j_1,j'_1\in [0,q)\cap \Z}}\EX_{H}\left(\sum_{\substack{ \bar{j}\neq\bar{j'}\\\bar{j},\bar{j'}\in [0,q)^{d-2}\cap \Z^{d-2}}}\left(\abs{\frac{j_{1}}{q}-\frac{j'_{1}}{q}}^2+\abs{f_{\theta}\left(\frac{j_{1}}{q}\right)-f_{\theta}\left(\frac{j'_{1}}{q}\right)}^2+\abs{\frac{\bar{j}}{q}-\frac{\bar{j'}}{q}}^2\right)^{-\frac{s}{2}}\right)
\end{equation}
where $\bar{j}=(j_2,...,j_{d-1})$. We just need to worry about bounding term $I$ since it is greater than term $II$. To see this, note that any addend in the inner sum of $II$ is comparable to an addend of the inner sum of $II$ but, at the same time, the outer sum of $I$ contains way more terms than the outer sum of $II$ (due to the difference between the conditions $j_1=j'_1$ and $j_1\neq j'_1$).
To bound the first term, we can further decompose it in two pieces $I=III+IV$ where:
\begin{equation}
III=q^{-2(d-1)}\sum_{\substack{j_{1}\neq j'_{1}\\j_1,j'_1\in [0,q)\cap \Z\\\abs{j_{1}-j'_{1}}\leq\frac{q}{2b^{2}}}}\EX_{H}\left(\sum_{\bar{j},\bar{j'}\in [0,q)^{d-2}\cap \Z^{d-2}}\left(\abs{\frac{j_{1}}{q}-\frac{j'_{1}}{q}}^2+\abs{f_{\theta}\left(\frac{j_{1}}{q}\right)-f_{\theta}\left(\frac{j'_{1}}{q}\right)}^2+\abs{\frac{\bar{j}}{q}-\frac{\bar{j'}}{q}}^2\right)^{-\frac{s}{2}}\right)
\end{equation}
\begin{equation}
IV=q^{-2(d-1)}\sum_{\substack{j_{1}\neq j'_{1}\\j_1,j'_1\in [0,q)\cap \Z\\\abs{j_{1}-j'_{1}}>\frac{q}{2b^{2}}}}\EX_{H}\left(\sum_{\bar{j},\bar{j'}\in [0,q)^{d-2}\cap \Z^{d-2}}\left(\abs{\frac{j_{1}}{q}-\frac{j'_{1}}{q}}^2+\abs{f_{\theta}\left(\frac{j_{1}}{q}\right)-f_{\theta}\left(\frac{j'_{1}}{q}\right)}^2+\abs{\frac{\bar{j}}{q}-\frac{\bar{j'}}{q}}^2\right)^{-\frac{s}{2}}\right)
\end{equation}
The last term is bounded as follows:
\begin{equation}
\begin{split}
IV&\leq q^{-2(d-1)}\sum_{\substack{j_{1}\neq j'_{1}\\j_1,j'_1\in [0,q)\cap \Z\\\abs{j_{1}-j'_{1}}>\frac{q}{2b^{2}}}}\sum_{\bar{j},\bar{j'}\in [0,q)^{d-2}\cap \Z^{d-2}}\abs{\frac{j_{1}}{q}-\frac{j'_{1}}{q}}^{-s}\\
&\leq q^{-2(d-1)}q^2q^{2(d-2)}(2b^2)^s=(2b^2)^s
\end{split}
\end{equation}
In the inner sum of $III$, the quantity $\abs{\frac{j_{1}}{q}-\frac{j'_{1}}{q}}^2+\abs{f_{\theta}(\frac{j_{1}}{q})-f_{\theta}(\frac{j'_{1}}{q})}^2$ is constant (i.e. does not depend on $\bar{j},\bar{j'}$). For the sake of clarity, set $\beta^2=\abs{\frac{j_{1}}{q}-\frac{j'_{1}}{q}}^2+\abs{f_{\theta}(\frac{j_{1}}{q})-f_{\theta}(\frac{j'_{1}}{q})}^2$. We can then find an upper bound for such inner sum:\\
\begin{equation}\label{inner sum estimate}
\begin{split}
\sum_{\bar{j},\bar{j'}\in [0,q)^{d-2}\cap \Z^{d-2}}\left(\beta^2+\abs{\frac{\bar{j}}{q}-\frac{\bar{j'}}{q}}^2\right)^{-\frac{s}{2}}
& =q^{d-2}\beta^{-s}+\sum_{\substack{\bar{j}\neq \bar{j'}\\\bar{j},\bar{j'}\in [0,q)^{d-2}\cap \Z^{d-2}}}\left(\beta^2+\abs{\frac{\bar{j}}{q}-\frac{\bar{j'}}{q}}^2\right)^{-\frac{s}{2}}\\
& \leq q^{d-2}\beta^{-s}+q^{d-2+s}\sum_{\substack{\bar{j}\neq \bar{j'}\\\bar{j}\in (-q,q)^{d-2}\cap \Z^{d-2}}}\left((q\beta)^2+\abs{\bar{j}}^2\right)^{-\frac{s}{2}} \\
& \lesssim q^{d-2}\beta^{-s}+q^{d-2+s}\int_{x\in [0,q)^{d-2}}\left((q\beta)^2+\abs{x}^2\right)^{-\frac{s}{2}}\,dx\\
& \lesssim q^{d-2}\beta^{-s}+q^{d-2+s}(q\beta)^{d-2-s}\int_{x\in \R^{d-2}}\left(1+\abs{x}^2\right)^{-\frac{s}{2}}\,dx\\
& \lesssim q^{d-2}\beta^{-s}+q^{d-2+s}(q\beta)^{d-2-s}\\
& \lesssim q^{2(d-2)}\beta^{d-2-s}
\end{split}
\end{equation}
In evaluating the integral, the change of variable $x$ to $(q\beta)x$ is performed.
Note that the above integral is convergent for any value $s>d-2$ (which is fine since we consider $s$ such that $d-1<s<d-\alpha$).
The first inequality follows by the change of variable $\bar{j}-\bar{j'}$ to $\bar{j}$ and the second one by integral test. Additionally, since we consider $j_1 \neq j'_1$ and then $\beta^2\geq q^{-2}$, the last inequality follows because:
\begin{equation}
\begin{split}
& q^{d-2+s}(q\beta)^{d-2-s} = \beta^{-s}q^{d-2+s}(q)^{d-2-s}(\beta)^{d-2} \\
& = \beta^{-s}q^{2(d-2)}(\beta)^{d-2}\geq\beta^{-s}q^{2(d-2)}(q^{-1})^{d-2}\\
& = \beta^{-s}q^{(d-2)}
\end{split}
\end{equation}
for $d\geq2$.

Consequently, using \eqref{inner sum estimate} and by substituting back for $\beta$:
\begin{equation}
\begin{split}
III & \lesssim q^{-2(d-1)}q^{2(d-2)}\sum_{\substack{j_{1}\neq j'_{1}\\j_1,j'_1\in [0,q)\cap \Z\\\abs{j_{1}-j'_{1}}\leq\frac{q}{2b^{2}}}}\EX_{H}\left(\left(\abs{\frac{j_{1}}{q}-\frac{j'_{1}}{q}}^2+\abs{f_{\theta}\left(\frac{j_{1}}{q}\right)-f_{\theta}\left(\frac{j'_{1}}{q}\right)}^2\right)^{\frac{d-2-s}{2}}\right)\\
& =q^{-2}\sum_{\substack{j_{1}\neq j'_{1}\\j_1,j'_1\in [0,q)\cap \Z\\\abs{j_{1}-j'_{1}}\leq\frac{q}{2b^{2}}}}\EX_{H}\left(\abs{\left(\frac{j_1}{q},f_{\theta}\left(\frac{j_1}{q}\right)\right)-\left(\frac{j'_1}{q},f_{\theta}\left(\frac{j'_1}{q}\right)\right)}^{d-2-s}\right)\\
& \lesssim q^{-2}q^{-1+\alpha-d+2+s}\sum_{\substack{j_{1}\neq j'_{1}\\j_1,j'_1\in [0,q)\cap \Z}}\abs{j_1-j'_1}^{1-\alpha+d-2-s}\\
& \lesssim q^{-1+\alpha-d+s}q\sum_{j_1\in[0,q)\cap\Z}\abs{j_1}^{\alpha+d-1-s}\lesssim q^{-d+\alpha+s}q^{d-\alpha-s}
\lesssim 1
\end{split}
\end{equation}
completing the proof also for the case $d\geq3$. The second inequality follows thanks to proposition $\ref{Proposition}$, the third one by the change of variable $j_1-j'_1$ to $j_1$ and the forth one by integral test (where the integral converges for values $s$ strictly smaller than $d-\alpha$, as desired).  Note that we could use proposition $\ref{Proposition}$ successfully since $1=-d+(d-1)+2<-d+s+2<-d+(d-\alpha)+2=2-\alpha$.  Again, the proof is complete since, for any given value $x\in(0,1]$ we can find values of $a$ and $b$ such that $\alpha=x$. This is a consequence of the fact that $\alpha=-\frac{\log(a)}{\log{b}}$ and $0<a<1<b$ with $ab\geq1$. In fact, fixing a positive $a<1$, $b=\frac{1}{a^{x}}\geq \frac{1}{a}$ for any value of $x\in (0,1]$ and $\alpha=x$.
\end{proof}

\vskip.25in 

\subsection{Proof of Theorem \ref{conor}} 

We will proceed by approximating the sum $|A_k|^{-2}\sum_{a\neq a'}||a-a'||^{-s}$ by using sets centered at each $a'\in A_k$ such that for each point $a$ in a given set, $|a_i-a'_i|\approx n_i^{-j_i}$ for all $i$, where $j_1,\dots,j_d$ all range from 1 to $k$. In this case, $|a_i-a'_i|\approx{n_i}^{-j_i}$ means that ${n_i}^{-j_i-1}<|a_i-a'_i|\leq{n_i}^{-j_i}$. We do not need to consider side lengths with dimensions $n_{i}^{-j_i}$ for $j_i>k$ because by construction, the $i$-th discrete Cantor set $C_{m_i,n_i}^k$ does not have any distinct points whose distance is less than $n_i^{-k}$. Approximating the sum this way gives us
\begin{align}
|A_k|^{-2}\sum_{a\neq a'}|a-a'|^{-s}\approx|A_k|^{-2}\sum_{j_1,\dots,j_d=1}^k\sum_{\substack{a\neq a'\\
                       |a_1-a'_1|\approx n_1^{-j_1}\\
                       \vdots\\
                       |a_d-a'_d|\approx n_d^{-j_d}}}|a-a'|^{-s}.
\end{align}
\\
Next, note that for any $a\neq a'$, 
\begin{align}
|a-a'|^{-s}\approx(|a_1-a'_1|+\cdots+|a_d-a'_d|)^{-s}\leq\max_{1\leq i\leq n}\{|a_i-a'_i|\}^{-s}.
\end{align}
Putting (4.1) and (4.2) together, we get
\begin{align}
|A_k|^{-2}\sum_{j_1,\dots,j_d=1}^k\sum_{\substack{a\neq a'\\
                       |a_1-a'_1|
                       \approx n_1^{-j_1}\\
                       \vdots\\
                       |a_d-a'_d|\approx n_d^{-j_d}}}|a-a'|^{-s}
                       &\approx
                       |A_k|^{-2}\sum_{j_1,\dots,j_d=1}^k\sum_{\substack{a\neq a'\\
                       |a_1-a'_1|\approx n_1^{-j_1}\\
                       \vdots\\
                       |a_d-a'_d|\approx n_d^{-j_d}}}\max_{1\leq i\leq n}\{|a_i-a'_i|\}^{-s}\\
                       &= |A_k|^{-2}\sum_{j_1,\dots,j_d=1}^k\sum_{\substack{a\neq a'\\
                       |a_1-a'_1|\approx n_1^{-j_1}\\
                       \vdots\\
                       |a_d-a'_d|\approx n_d^{-j_d}}}\max_{1\leq i\leq n}\{n_i^{-j_i}\}^{-s}.
\end{align}
Temporarily fix $a'\in A_k$. For any $i$, the number of elements $a_i\in C_{m_i,n_i}^k$ such that $|a_i-a'_i|\approx n_i^{-j_i}$ is approximately $\frac{|C_{m_i,n_i}^k|}{m_i^{j_i}}$. Therefore, it follows that the number of $a\in A_k$ such that $|a_i-a'_i|\approx n_i^{-j_i}$ for all $i$ is
$$\frac{|C_{m_1,n_1}^k|}{m_1^{j_1}}\cdots\frac{|C_{m_d,n_d}^k|}{m_d^{j_d}}=\frac{|A_k|}{m_1^{j_1}\cdots m_d^{j_d}}.$$
Finally, since there are $|A_k|$ possible choices for $a'$, we can see that
\begin{align}
|A_k|^{-2}\sum_{j_1,\dots,j_d=1}^k\sum_{\substack{a\neq a'\\
                       +|a_1-a'_1|\approx n_1^{-j_1}\\
                       \vdots\\
                       |a_d-a'_d|\approx n_d^{-j_d}}}\max_{1\leq i\leq n}\{n_i^{-j_i}\}^{-s}
                       &= |A_k|^{-2}\sum_{j_1,\dots,j_d=1}^k\max_{1\leq i\leq n}\{n_i^{-j_i}\}^{-s}\sum_{\substack{a\neq a'\\
                       +|a_1-a'_1|\approx n_1^{-j_1}\\
                       \vdots\\
                       |a_d-a'_d|\approx n_d^{-j_d}}}1\\
                       &\approx |A_k|^{-2}|A_k||A_k|\sum_{j_1,\dots,j_d=1}^km_1^{-j_1}\cdots m_d^{-j_d}\max_{1\leq i\leq n}\{n_i^{-j_i}\}^{-s}\\
                       &= \sum_{j_1,\dots,j_d=1}^km_1^{-j_1}\cdots m_d^{-j_d}\max_{1\leq i\leq n}\{n_i^{-j_i}\}^{-s}
\end{align}
\\
Now, consider the case where $\max_{1\leq i\leq n}\{n_i^{-j_i}\}=n_1^{-j_1}$. This would mean that for all $i$, 
\begin{align}
n_1^{-j_1}\geq n_i^{-j_i}\implies n_1^{j_1}\leq n_i^{j_i}\implies j_1\frac{\ln(n_1)}{\ln(n_i)}\leq j_i. 
\end{align}
Note that the inequalities in (4.5) also hold for $n_i^{-j_i}$ when it is maximal. Therefore, we can further split the inner sum in (4.4) by considering the different cases in which each $n_i^{-j_i}$ is maximal to get
\begin{align*}
\sum_{j_1,\dots,j_d=1}^km_1^{-j_1}\cdots m_d^{-j_d}\max_{1\leq i\leq n}\{n_i^{-j_i}\}^{-s}
                       &= \sum_{j_1\frac{\ln(n_1)}{\ln(n_i)}\leq j_i\leq k}m_1^{-j_1}\cdots m_d^{-j_d}n_1^{j_1s}+\cdots\\
                       &\textbf{ }\,\,\,\,+\sum_{j_d\frac{\ln(n_d)}{\ln(n_i)}\leq j_i\leq k}m_1^{-j_1}\cdots m_d^{-j_d}n_d^{j_1s}.
\end{align*}
We restrict our focus to the first term of the above sum in order to determine for which $s$ the sum converges. For $2\leq i\leq n$, the term $n_i^{-j_i}$ in the sum is a geometric series starting at $j_1\frac{\ln(n_1)}{\ln(n_i)}$, so we can approximate each of these terms by $m_i^{j_1\frac{\ln(n_1)}{\ln(n_i)}}$. This gives us that
\begin{align*}
\sum_{j_1\frac{\ln(n_1)}{\ln(n_i)}\leq j_i\leq k}m_1^{-j_1}\cdots m_d^{-j_d}n_1^{j_1s}
    &\approx \sum_{j_1=1}^nm_1^{-j_1}\cdots m_d^{-j_1\frac{\ln(n_1)}{\ln(n_d)}}n_1^{j_1s}
    = \sum_{j_1=1}^n\left(m_1^{-1}\cdots m_d^{-\frac{\ln(n_1)}{\ln(n_d)}}n_1^{s}\right)^{j_1}.
\end{align*}
Now we have a geometric series, so we know that for this series to converge as $n\to\infty$, we must have that
\begin{align*}
m_1^{-1}\cdots m_d^{-\frac{\ln(n_1)}{\ln(n_d)}}n_1^{s}
    <1
    &\iff -\left(\ln(m_1)+\ln(m_2)\frac{\ln(n_1)}{\ln(n_2)}+\cdots+\ln(m_d)\frac{\ln(n_1)}{\ln(n_d)}\right)+s\ln(n_1)<0\\
    &\iff s<\frac{\ln(m_1)}{\ln(n_1)}+\cdots+\frac{\ln(m_d)}{\ln(n_d)}.
\end{align*}
We can use an identical argument to show that all of the other sums $$\sum_{j_l\frac{\ln(n_1)}{\ln(n_i)}\leq j_i\leq k}m_1^{-j_1}\cdots m_d^{-j_d}n_l^{j_ls}$$ also converge exactly when $$s<\frac{\ln(m_1)}{\ln(n_1)}+\cdots+\frac{\ln(m_d)}{\ln(n_d)}.$$ 
\\
Furthermore, since $$I_s(A_k)=|A_k|^{-2}\sum_{a\neq a'}|a-a'|^{-s}\approx\sum_{j_1\frac{\ln(n_1)}{\ln(n_i)}\leq j_i\leq k}m_1^{-j_1}\cdots m_d^{-j_d}n_1^{j_1s}+\cdots+\sum_{j_d\frac{\ln(n_d)}{\ln(n_i)}\leq j_i\leq k}m_1^{-j_1}\cdots m_d^{-j_d}n_d^{j_1s},$$ this implies that $\{I_s(A_k)\}_{k\geq1}$ converges exactly when $$s<\frac{\ln(m_1)}{\ln(n_1)}+\cdots+\frac{\ln(m_d)}{\ln(n_d)}.$$ Hence, $\{A_k\}_{k\geq1}$ is $s$-adaptable for exactly these values of $s$, so by definition, we may conclude that $$s_{critical}=\frac{\ln(m_1)}{\ln(n_1)}+\cdots+\frac{\ln(m_d)}{\ln(n_d)}$$ as desired.

With this result, if we revisit the discrete Cantor set products from before, then we get that the dimension of $\{C_{2,3}^k\times C_{2,3}^k\}_{k\geq1}$ is $\frac{\ln(2)}{\ln(3)}+\frac{\ln(2)}{\ln(3)}=2\frac{\ln(2)}{\ln(3)}$, the dimension of $\{C_{2,4}^k\times C_{2,4}^k\}_{k\geq1}$ is $\frac{\ln(2)}{\ln(4)}+\frac{\ln(2)}{\ln(4)}=1$ and the dimension of $\{C_{2,4}^k\times C_{2,3}^k\times C_{2,3}^k\}_{k\geq1}$ is $\frac{\ln(2)}{\ln(4)}+\frac{\ln(2)}{\ln(3)}+\frac{\ln(2)}{\ln(3)}=\frac{1}{2}+2\frac{\ln(2)}{\ln(3)}$, which makes sense given how sparse these sets are. This demonstrates that our alternate notion of dimension is better at capturing the dimensionality of discrete Cantor set products than PCA.

\end{document}